\documentclass[12pt]{amsart}

\usepackage{amssymb}

\usepackage{enumerate}

\usepackage{graphicx}

\usepackage{amsmath}
\usepackage{amsfonts}
\usepackage{amsmath}
\usepackage{amsthm}
\usepackage{amssymb}
\usepackage{amscd}
\usepackage{mathrsfs}
\usepackage{latexsym}
\usepackage{times}
\usepackage{fancyhdr}
\usepackage{epsfig}
\usepackage[all]{xy}
\usepackage{xypic}
\usepackage[english]{babel}
\usepackage{tkz-graph}

\usepackage{times}
\usepackage{fancyhdr}
\usepackage{epsfig}
\usepackage[all]{xy}
\usepackage[english]{babel}
\usepackage{leftidx}
\usepackage{enumerate}
\usepackage{graphicx}
\usepackage{epigraph}

\usepackage{xypic}
\usepackage[english]{babel}
\usepackage{tkz-graph}
\usepackage{longtable}
\usepackage{supertabular}

\makeatletter
\@namedef{subjclassname@2010}{%
  \textup{2010} Mathematics Subject Classification}
\makeatother

\newtheorem{thm}{Theorem}[section]

\newtheorem{lem}[thm]{Lemma}
\newtheorem{prop}[thm]{Proposition}
\newtheorem{prob}[thm]{Problem}

\theoremstyle{definition}

\numberwithin{equation}{section}

\def\Gal{\text{Gal}}

\def\Hom{\text{Hom}}

\def\gcd{\text{gcd}}
\def\Spec{\text{Spec}}

\def\Im{\text{Im}}

\def\Gon{\text{Gon}}
\def\div{\text{div}}

\newcommand{\hooklongrightarrow}{\lhook\joinrel\longrightarrow}

\begin{document}


\baselineskip=17pt



\title[Cyclic torsion]{On the cyclic torsion of elliptic curves over cubic number fields.}

\author[J. Wang]{Jian Wang}
\address{Department of Mathematics\\ University of Southern California\\
Los Angeles, CA 90089, USA}

\address{Current address: Department of Mathematical Sciences\\ Tsinghua University\\
Beijing, 100084, China}
\email{blandye@gmail.com}

\date{\today}

\begin{abstract}
Let $E$ be an elliptic defined over a number field $K$. Then its Mordell-Weil group $E(K)$ is finitely generated:
$E(K)\cong E(K)_{tor}\times\mathbb{Z}^r$. In this paper, we discuss the cyclic torsion subgroup of elliptic curves over cubic number fields. For $N=169,143,91,65,77$ or $55$, we show that $\mathbb{Z}/N\mathbb{Z}$ is not a subgroup of $E(K)_{tor}$ for any elliptic curve $E$ over a cubic number field $K$.
\end{abstract}



\maketitle

\section{Introduction}

Let $E$ be an elliptic defined over a number field $K$. Then its Mordell-Weil group $E(K)$ is finitely generated:
$$E(K)\cong E(K)_{tor}\times\mathbb{Z}^r$$
For a fixed elliptic $E$ over a field $K$, the torsion component $E(K)_{tor}$ can be calculated due to the Nagell-Lutz-Cassels theorem \cite{Cassels}. However, if we consider a class of elliptic curves, it is usually difficult to list exactly all the possible group structures of $E(K)_{tor}$. The following problem is one of this kind.

\begin{prob}
For an integer $d\geq1$, what are the possible group structures of $E(K)_{tor}$ with $[K:\mathbb{Q}]=d$?
\end{prob}

For $d=1$, i.e. $K=\mathbb{Q}$, by the work of Kubert \cite{Kubert} and Mazur \cite{Mazur}, the torsion group $E(\mathbb{Q})_{tor}$ of an elliptic curve $E$ over the rational number field is isomorphic to one of the following:
$$\aligned&\mathbb{Z}/m\mathbb{Z},&m=1-10,12;\\&\mathbb{Z}/2\mathbb{Z}\times\mathbb{Z}/2m\mathbb{Z},&m=1-4.\endaligned$$

For $d=2$, by the work of Kenku-Momose \cite{KenkuMomose} and Kamienny \cite{Kamienny1}, the torsion group $E(K)_{tor}$ of an elliptic curve over a quadratic number field is isomorphic to one of the following:
$$\aligned&\mathbb{Z}/m\mathbb{Z},&m=1-16,18;
\\&\mathbb{Z}/2\mathbb{Z}\times\mathbb{Z}/2m\mathbb{Z},&m=1-6;
\\&\mathbb{Z}/3\mathbb{Z}\times\mathbb{Z}/3m\mathbb{Z},&m=1-2;
\\&\mathbb{Z}/4\mathbb{Z}\times\mathbb{Z}/4\mathbb{Z}.\endaligned$$

For $d=3$, Parent \cite{Parent2,Parent3} showed that the prime divisors of the order of $E(K)_{tor}$ are $\leq13$. Jeon-Kim-Schweizer \cite{JeonKimSchweizer} determined all the torsion structures that appear infinitely often when we run through all elliptic curves over all cubic fields:
$$\aligned&\mathbb{Z}/m\mathbb{Z},&m=1-16,18,20;\\&\mathbb{Z}/2\mathbb{Z}\times\mathbb{Z}/2m\mathbb{Z},&m=1-7.\endaligned$$
Najman \cite{Najman} discovered a \emph{sporadic} elliptic curve over a cubic field with torsion group isomorphic to $\mathbb{Z}/21\mathbb{Z}$. In view of these facts, our ultimate aim is to show that the torsion group $E(K)_{tor}$ of an elliptic curve $E$ over a cubic number field is isomorphic to one of the following:
$$\aligned&\mathbb{Z}/m\mathbb{Z},&m=1-16,18,20-21;\\&\mathbb{Z}/2\mathbb{Z}\times\mathbb{Z}/2m\mathbb{Z},&m=1-7.\endaligned$$
For the cyclic case, it suffices to show that $\mathbb{Z}/N\mathbb{Z}$ is not a subgroup of $E(K)_{tor}$ for any elliptic curve $E$ over a cubic number field $K$ when $N$ is among the following list
$$\aligned N&=169,121,49,25,27,32;\\
N&=143,91,65,39,26,77,55,33,22,35,63,42,28,45,30,40,36,24.\endaligned$$

The main result of this paper is the following:
\begin{thm}\label{T1} If $N=169,143,91,65,77$ or $55$, then $\mathbb{Z}/N\mathbb{Z}$ is not a subgroup of $E(K)_{tor}$ for any elliptic curve $E$ over a cubic number field $K$.
\end{thm}

\section{Preliminaries}

Let $\mathbb{H}=\{z\in\mathbb{C}~|~\Im z>0\}$ be the upper half plane. Let $\overline{\mathbb{H}}=\mathbb{H}\cup\mathbb{P}^1(\mathbb{Q})$ be the extended upper half plane by adjoining cusps $\mathbb{P}^1(\mathbb{Q})=\mathbb{Q}\cup\{\infty\}$ to $\mathbb{H}$. Let $N$ be a positive integer. Let
$$\aligned\Gamma_0(N)&=\left\{\left(
                          \begin{array}{cc}
                            a & b \\
                            c & d \\
                          \end{array}
                        \right)\in SL_2(\mathbb{Z})/(\pm1)|c\equiv0\mod N
\right\}\\
\Gamma_1(N)&=\left\{\left(
                          \begin{array}{cc}
                            a & b \\
                            c & d \\
                          \end{array}
                        \right)\in \Gamma_0(N)|a\equiv d\equiv1\mod N
\right\}
\endaligned$$
be the congruence subgroups and let $X_1(N)$ (resp. $X_0(N)$) be the modular curve which corresponds to the modular group $\Gamma_1(N)$ (resp. $\Gamma_0(N)$). We denote by $Y_1(N)=X_1(N)\backslash\{cusps\}$, $Y_0(N)=X_0(N)\backslash\{cusps\}$ the corresponding affine curves. Denote by $J_1(N)$ (respectively, $J_0(N)$) the jacobian of $X_1(N)$ (respectively, $X_0(N)$).

For a modular curve $X$, let $X^{(d)}$ be the $d$-th symmetric power of $X$, i.e. the quotient space of the $d$-fold product $X^d$ by the action of the symmetric group $S_d$ permuting the factors. Let $\overline{\mathbb{Q}}$ be the algebraic closure of $\mathbb{Q}$. Then $X^{(d)}(\overline{\mathbb{Q}})$, the set of algebraic points of $X^{(d)}$, corresponds one-to-one to the set $\{P_1+\cdots+P_d; P_i\in X(\overline{\mathbb{Q}})\}$ of positive $\overline{\mathbb{Q}}$-rational divisors of degree $d$ of $X$.

Let $K$ be a number field of degree $d$ over $\mathbb{Q}$. Let $x\in X(K)$. Let $x_1,\cdots,x_d$ be the images of $x$ under the distinct embeddings $\tau_i:K\hooklongrightarrow\mathbb{C}, 1\leq i\leq d$. We may view $x_1+\cdots+x_d$ is a $\mathbb{Q}$-rational point of $X^{(d)}$. Define $$\Phi:X^{(d)}\longrightarrow J_X$$
by $\Phi(P_1+\cdots+P_d)=[P_1+\cdots+P_d-d\infty]$ where $J_X$ is the jacobian of $X$, and $[~~]$ denotes the divisor class.

For a modular curve $X$ over $\mathbb{C}$, $X$ is called $d$-gonal if there exists a finite $\mathbb{C}$-morphism $\pi:X\longrightarrow \mathbb{P}_\mathbb{C}^1$ of degree $d$. The minimum $d$ is called the \emph{gonality} of $C$, which we denote as $\Gon(C)$. The following lemma is a generalization of proposition 1(i) in Frey \cite{Frey}.

\begin{lem}[Frey]\label{Frey} Assume that $Gon(X)>d$. Then $\Phi$ is injective.
\end{lem}

\begin{proof} Suppose otherwise $\Phi$ is not injective, i.e. there exists different $P_1+\cdots+P_d$ and $Q_1+\cdots+Q_d$ in $X^{(d)}$ such that $\Phi(P_1+\cdots+P_d)=\Phi(Q_1+\cdots+Q_d)$, then $$[P_1+\cdots+P_d-d\infty]=[Q_1+\cdots+Q_d-d\infty]\in J_X$$
then there is a nonconstant function $f\in \mathbb{C}(X)^*$ such that
$$\aligned\div(f)&=(P_1+\cdots+P_d-d\infty)-(Q_1+\cdots+Q_d-d\infty)
\\&=P_1+\cdots+P_d-Q_1-\cdots-Q_d\endaligned$$
which means $f$ has a pole divisor of degree $\leq d$. Consider the map $\pi:X\longrightarrow\mathbb{P}_{\mathbb{C}}^1$ defined by $P\longmapsto[f(P),1]$. Then the degree of $\pi$ is equal to the degree of pole divisor of $f$. This contradicts the assumption $\Gon(X)>d$.
\end{proof}

We are interested in the gonality of the modular curve $X_0(N)$ over $\mathbb{C}$.  Since the 1-gonal curves are precisely the curves of genus 0, then $X_0(N)$ is 1-gonal if and only if $N$ is among the fifteen values $N=1-10,12,13,16,18,25$ with genus 0. The complete list of 2-gonal $X_0(N)$ was determined by Ogg \cite{Ogg1}, and that of 3-gonal ones by Hasegawa-Shimura \cite{HasegawaShimura}.

\begin{prop}[Ogg]\label{Ogg} The modular curve $X_0(N)$ is 2-gonal if and only if $N$ is one of the following:
$$\aligned N&=1-10,12,13,16,18,25   &(g=0);\\
N&=11,14,15,17,19-21,24,27,32,36,49   &(g=1);\\
N&=22,23,26,28,29,31,37,50   &(g=2);\\
N&=30,33,35,39,40,41,48   &(g=3);\\
N&=47     &(g=4);\\
N&=46,59   &(g=5);\\
N&=71   &(g=6).\endaligned$$
\end{prop}

\begin{prop}[Hasegawa-Shimura]\label{HasegawaShimura} The modular curve $X_0(N)$ is 3-gonal if and only if $N$ is one of the following:
$$\aligned N&=1-10,12,13,16,18,25   &(g=0);\\
N&=11,14,15,17,19-21,24,27,32,36,49   &(g=1);\\
N&=22,23,26,28,29,31,37,50   &(g=2);\\
N&=34,43,45,64   &(g=3);\\
N&=38,44,53,54,61,81 &(g=4).\endaligned$$
\end{prop}

The moduli interpretation of noncuspidal points of $X_1(N)$ are $(E,\pm P)$, where $E$ is an elliptic curve and $P\in E$ is a point of order $N$.  The moduli interpretation of noncuspidal points of $X_0(N)$ are $(E, C)$, where $E$ is an elliptic curve and $C\subset E$ is a cyclic subgroup of order $N$. The map $\pi: X_1(N)\longrightarrow X_0(N)$ send $(E,\pm P)$ to $(E,\langle P\rangle)$, where $\langle P\rangle$ is the cyclic subgroup generated by $P$.

Let $p$ be a prime such that $p\nmid N$. Igusa's theorem \cite{Igusa} says that the modular curves $X_1(N)$ and $X_0(N)$ have good reduction at prime $p$. Moreover, reducing the modular curve is compatible with reducing the moduli interpretation (See for example \cite[Theorem 1]{Ogg75}). And the description of the cusps is the same in characteristic $p$ as in characteristic $0$.

Let $k=\mathbb{F}_q$ be the finite field with $q=p^n$ elements. Let $E/k$ be an elliptic curve over $k$. Let $|E(k)|$ be the number of points of $E$ over $k$. Then Hasse's theorem states that
$$||E(k)|-q-1|\leq2\sqrt{q}$$
i.e.
$$(1-\sqrt{p^n})^2\leq|E(k)|\leq(1+\sqrt{p^n})^2$$

The description of the reduction types of elliptic curves in terms of the language of  N\'{e}ron models can be summarised as the Kodaira-N\'{e}ron theorem \cite{Kodaira} \cite{Neron}. A complete proof of this theorem can be found in \cite[IV \S8 \S9]{Silverman2}.

\begin{thm}[Kodaira-N\'{e}ron] Let $R$ be a Dedekind domain with field of fractions
$K$, let $\mathcal{E}$ be a N\'{e}ron model over $R$ for an elliptic curve $E/K$, and let $\wp\subset R$
be any nonzero prime ideal with residue field $k$. Let $\widetilde{E}$ be the fibre over $k$ of $\mathcal{E}$.

(1): If $E$ has good reduction at $\wp$,
then $\widetilde{E}(k)=\widetilde{E}(k)^0$ is an elliptic curve, where $\widetilde{E}(k)^0$ denotes the connected component of the identity.

(2): If $E$ has additive reduction at $\wp$,
then $\widetilde{E}(k)^0\cong\mathbb{G}_{a/k}$ , and $\widetilde{E}(k)/\widetilde{E}(k)^0=G$ is a finite group of order at most four.

(3): If $E$ has multiplicative reduction at $\wp$, then
there exists an extension $\mathscr{K}$ of $k$ of degree at most two so that $\widetilde{E}(\mathscr{K})^0\cong\mathbb{G}_{m/\mathscr{K}}$ and
$\widetilde{E}(\mathscr{K})/\widetilde{E}(\mathscr{K})^0\cong \mathbb{Z}/n\mathbb{Z}$ for some positive integer $n$. \end{thm}

Let $K$ be a number field with ring of integers $\mathcal{O}_K$, $\wp\subset\mathcal{O}_K$ a prime ideal lying above $p$, $k=\mathbb{F}_q=\mathcal{O}_K/\wp$ its residue field. Let $E$ be an elliptic curve over $K$ and $P\in E(K)$ a point of order $N$. Let $\widetilde{E}$ be the fibre over $k$ of the N\'{e}ron model of $E$, and let $\widetilde{P}\in\widetilde{E}(k)$ be the reduction of $P$. Suppose that $p\nmid N$. Then elementary theory of group schemes shows that $\widetilde{P}$ has order $N$ due to the following well-known result (See for example \cite[\S 7.3 Proposition 3]{BoschLutkebohmertRaynaud}).

\begin{prop}\label{torsionreduction} Let $m$ be a positive integer relatively prime to $char(k)$. Then the reduction map
$$E(K)[m]\longrightarrow\widetilde{E}(k)$$
is injective.
\end{prop}

The N\'{e}ron-Kodaira theorem lead to Deligne-Rapoport's treatment of modular curves as moduli scheme of generalized elliptic curves \cite{DeligneRapoport}. Katz-Mazur \cite{KatzMazur} developed the theory of Drinfeld level structures on elliptic curves. Conrad \cite{Conrad} improved this theory by extending it on to generalized elliptic curves. We explain here the notions and results in these theories that are necessary in Section \ref{method}.

Let $n\geq1$ be an integer and let $k$ be a field. The \emph{N\'{e}ron $n$-gon} over $k$, denoted $C_n$, is the quotient of $(\mathbb{P}^1)_k\times\mathbb{Z}/n\mathbb{Z}$ where $(\infty,i)$ is identified with $(0,i+1)$. It has $n$ irreducible components $(\mathbb{P}^1)_k\times d, d\in\mathbb{Z}/n\mathbb{Z}$, of which $(\mathbb{P}^1)_k\times 0$ is called the identity component. The smooth locus $C^{sm}_n=\mathbb{G}_m\times\mathbb{Z}/n\mathbb{Z}$ of $C_n$ is a group. Furthermore, the action of $C^{sm}_n$ on itself extends to an action of $C^{sm}_n$ on all of $C_n$: the $\mathbb{G}_m$ part fixes the singular points. The $N$-torsion part $C^{sm}_n[n]$ has order $n^2$. In fact, there is a natural short exact sequence
$$0\longrightarrow\mu_n\longrightarrow C^{sm}_n[n]\longrightarrow\mathbb{Z}/n\mathbb{Z}\longrightarrow0$$
where the $\mu_n$ sits in the identity component of $C^{sm}_n$.

A \emph{generalized elliptic curve} over a base scheme $S$ is a tuple $(E,+,e)$, where $E/S$ is a proper flat curve, $e\in E(S)$, and $+$ is a map $E^{sm}\times E\longrightarrow E$ such that: (1) $+$ (with $e$) gives $E^{sm}$ the structure of a group and defines an action on $E$; (2) the geometric fibers of $E$ are elliptic curves or N\'{e}ron $n$-gons.

Denote $S=\Spec~\mathbb{Z}$. For $N\geq5$, $X_1(N)_{/S}$ is the fine moduli scheme which classify the generalized elliptic curves $E$ with a torsion point $P$ of order $N$; $X_0(N)_{/S}$ is the coarse moduli scheme which classify the generalized elliptic curves $E$ with a cyclic subgroup $C$ of order $N$. There is a natural morphism $X_1(N)_{/S}\longrightarrow X_0(N)_{/S}: (E,\pm P)\longmapsto(E,\langle P\rangle)$, where $\langle P\rangle$ is the cyclic subgroup generated by $P$.

Now we can describe the moduli interpretation of the cusps on the generic fiber $X_1(N)$ (resp. $X_0(N)$) of $X_1(N)_{/S}$ (resp. $X_0(N)_{/S}$). The moduli interpretation of cusps of $X_1(N)$ is that for each $d~|~N$, one has cusps $(C_d,(\zeta_N^r,b))$ where $b\in(\mathbb{Z}/d\mathbb{Z})^\times$ and $r\in\mathbb{Z}/N\mathbb{Z}$ maps to a unit in $\mathbb{Z}/(N/d)\mathbb{Z}$. It is easy to see $(\zeta_N^r,b)$ is a point of order $N$ in the smooth locus $C_d^{sm}=\mathbb{G}_m\times\mathbb{Z}/d\mathbb{Z}$. The moduli interpretation of cusps of $X_0(N)$ is that for each $d~|~N$, one has cusps $(C_d, G)$, where $G$ is a cyclic subgroup of order $N$ in the smooth locus $C_d^{sm}=\mathbb{G}_m\times\mathbb{Z}/d\mathbb{Z}$ that meets all the irreducible components. Especially for $d=1$ and $d=N$, we have the cusps $(C_1,\mu_N)$ and $(C_N,\mathbb{Z}/N\mathbb{Z})$, which we denote as $0$ and $\infty$ respectively. Note that $0$ is distinguished from $\infty$ by the fact that $\mu_N$ lives in the identity component.

In the following section, we use a specialization lemma in Appendix of Katz \cite{Katz} and a theorem of Manin \cite{Manin} and Drinfeld \cite{Drinfeld}.

\begin{lem}[Specialization Lemma]\label{Katz} Let $K$ be a number field. Let $\wp\subset\mathcal{O}_K$ be a prime above $p$. Let $A/K$ be an abelian variety. Suppose the ramification index $e_\wp(K/\mathbb{Q})<p-1$. Then the reduction map
$$\Psi: A(K)_{tor}\longrightarrow A(\overline{\mathbb{F}}_p)$$
is injective.
\end{lem}

\begin{thm}[Manin-Drinfeld]\label{ManinDrinfeld} Let $C\subset SL_2(\mathbb{Z})/(\pm1)$ be a congruence subgroup. $x,y\in\mathbb{P}^1(\mathbb{Q})$ and $\overline{x},\overline{y}$ the images of $x$ and $y$ respectively, on $\overline{\mathbb{H}}/C$. Then the class of divisors $(\overline{x})-(\overline{y})$ on curve $\overline{\mathbb{H}}/C$ has finite order.
\end{thm}


\section{Method}\label{method}

When $N$ is a rational prime number, Kamienny \cite{Kamienny2} established an criterion for the nonexistence of elliptic curves $E$ with a point of order $N$ over a number field of degree $d$. This criterion is refined by Merel \cite{Merel} in which the Eisenstein quotient is replaced by the winding quotient and the linear independence condition of weight-two cusp forms is replaced by the linear independence of the Hecke operators on the winding element. This type of Kamienny's criterion for the general $N$ is proved by Parent \cite{Parent1}.(In Parent's paper, he assumed $N$ to be a prime power for practical reason. But as he mentioned on page 86, Th\'{e}or\`{e}me 1.7 and the Kamienny's criterion Th\'{e}or\`{e}me 1.8 are also true by taking directly at any positive integer level $N$.) Before giving this criterion, we have to explain the necessary knowledge.

Considering the first absolute singular homology group $H_1(X_0(N);\mathbb{Z})$ and the homology group relative to the cusps $H_1(X_0(N),cusps;\mathbb{Z})$ of $X_0(N)$, the first being seen as a subgroup of the second. For $(\alpha,\beta)\in\mathbb{P}^1(\mathbb{Q})^2$, the \emph{modular symbol} $\{\alpha, \beta\}$ is the element of $H_1(X_0(N),cusps;\mathbb{Z})$ defined by the image in $X_0(N)$ of geodesic path of $\mathbb{H}$ connecting $\alpha$ to $\beta$ in $\mathbb{H}\cup\mathbb{P}^1(\mathbb{Q})$. When $\Gamma_0(N)\alpha=\Gamma_0(N)\beta$, we have $\{\alpha,\beta\}\in H_1(X_0(N);\mathbb{Z})$. Integration defines a classical isomorphism of real vector spaces:
$$\aligned H_1(X_0(N);\mathbb{Z})\otimes\mathbb{R}&\longrightarrow \Hom_\mathbb{C}(H^0(X_0(N);\Omega^1),\mathbb{C})
\\\gamma\otimes1&\longmapsto(\omega\longmapsto\int_\gamma\omega)\endaligned$$
The following lemma is a generalization of Lemma 18.6 in Mazur \cite{Mazur}.

\begin{lem}\label{Mazur} The inverse image $e$ of the linear form
$$\omega\longmapsto\int_{\{0,\infty\}}\omega$$
in $H_1(X_0(N);\mathbb{Z})\otimes\mathbb{R}$ is actually in $H_1(X_0(N);\mathbb{Z})\otimes\mathbb{Q}$.
\end{lem}

\begin{proof} Consider the exact sequence of topological groups:
$$\xymatrix{0\ar[r]& H_1(X_0(N);\mathbb{Z})\ar[r]& U\ar[r]^{\pi~~~~~} &J_0(N)\ar[r]&0}$$
where $U$ is the universal covering group of the jacobian $J_0(N)$ of $X_0(N)$. As a real Lie group, $U$ is isomorphic to $H_1(X_0(N);\mathbb{Z})\otimes\mathbb{R}$ and $J_0(N)$ is canonically isomorphic to $H_1(X_0(N);\mathbb{Z})\otimes(\mathbb{R}/\mathbb{Z})$.  From the definition, it is clear that $\pi(e)=c=((0)-(\infty))$ in $J_0(N)$. By Theorem \ref{ManinDrinfeld}, $c$ has finite order, i.e. there is $n\in\mathbb{Z}_{\geq0}$ such that $n\cdot c=0$. It follows that $n\cdot e\in H_1(X_0(N);\mathbb{Z})$. So $e\in (1/n)H_1(X_0(N);\mathbb{Z})\subset H_1(X_0(N);\mathbb{Z})\otimes\mathbb{Q}$.
\end{proof}

This element $e$ in Lemma \ref{Mazur} was first defined by Mazur \cite{Mazur} as the \emph{winding element}. Denote $\mathbb{T}$ the algebra generated over $\mathbb{Z}$
by the Hecke operators $T_i$ ($i\geq1$, integer), acting faithfully on $H_1(X_0(N);\mathbb{Z})\otimes\mathbb{Q}$ and on the Jacobian $J_0(N)$ of the modular curve. Let $\mathscr{A}_e$ be the annihilator ideal of $e$ in $\mathbb{T}$; we then define the \emph{winding quotient} $J_0^e$ as the quotient abelian variety $J_0(N)/\mathscr{A}_eJ_0(N)$. Parent\cite[Theorem 1.7]{Parent1} showed that $J_0^e(\mathbb{Q})$ is finite.

The notion of \emph{formal immersion} was introduced by Mazur \cite{Mazur78} to indicate the morphism that satisfies the equivalent conditions of EGA IV Proposition 17.4.4. \cite{Grothendieck}. If $f:X\longrightarrow Y$ is a morphism of finite type between noetherian schemes, we shall say that $f$ is a \emph{formal immersion} at a point $x$ if the induced map on the completions of local rings $\widehat{f}^\sharp: \widehat{\mathcal{O}}_{Y,f(x)}\longrightarrow\widehat{\mathcal{O}}_{X,x}$ is surjective.

The following Lemma was known for experts but used without proof. Parent \cite{Parent1} refered it to an unpublished paper of Oesterl\'e \cite{Oesterle}. Arnold sketched a proof in a note \cite{Arnold}. For the sake of completeness, we write down his proof with more detailed clarification.

\begin{lem} \label{Oesterle} Suppose that $X$ is separated and that $f: X\longrightarrow Y$ is a formal
immersion at $x\in X$. Suppose that there is an integral Noetherian scheme $T$
and two $T$-valued points $p_1, p_2\in X(T)$ such that for some point $t\in T$ we have
$x=p_1(t)=p_2(t)$. If moreover $f\circ p_1=f\circ p_2$, then $p_1=p_2$.
\end{lem}

\begin{proof} The subscheme $A=\{s\in T ~|~ p_1(s)=p_2(s)\}\subseteq T$ is closed since $X$ is
separated. This is because in the following diagram $i$ is a base change of $\Delta$. So $i$ is a closed immersion since $\Delta$ is a closed immersion and the property of closed immersion is stable under base change.
$$\xymatrix {A=X\times_{X\times X}T\ar[r]^{~~~~~~~~~~i}\ar[d]&T\ar[d]^{p_1\times p_2}\\
X\ar[r]^\Delta&X\times X }
$$

We consider the canonical morphisms \cite[EGA I, \S 2.4]{Grothendieck1}
$$\phi_{T,t}: \Spec\mathcal{O}_{T,t}\longrightarrow T,~~~~~~~~\phi_{X,x}: \Spec\mathcal{O}_{X,x}\longrightarrow X.$$
By \cite[EGA I, Proposition 2.4.2]{Grothendieck1}, they are monomorphisms of ringed spaces. The image of $\phi_{T,t}$ (resp. $\phi_{X,x}$) is exactly the set of all those generic points of the closed irreducible subschemes of $T$ (resp. $X$) passing through $t$ (resp. $x$).

Since $T$ is integral, then $T=\overline{(0)}$, where $(0)$ is the unique generic point of $T$. So we will have an inclusion sequence
$$\overline{(0)}\subseteq\overline{\Spec\mathcal{O}_{T,t}}\subseteq A\subseteq T=\overline{(0)}$$
if we can show that $\Spec\mathcal{O}_{T,t}\longrightarrow T$ factors
through $A$. Hence we can assume that $T$ is local with closed point $t$. The maps
$p_i:T\longrightarrow X$ then factor (uniquely) through $\Spec\mathcal{O}_{X,x}\longrightarrow X$ , so we may assume
that $X$ is local with closed point $x$. Now we have the following commutative diagram
$$\xymatrix{&\Spec\mathcal{O}_{T,t}\ar[d]_{\phi_{T,t}}\ar@<-.5ex>[r]_{p_2}\ar@<.5ex>[r]^{p_1}&\Spec\mathcal{O}_{X,x}\ar[d]^{\phi_{X,x}}\\
A\ar[r]^{i}&T\ar@<-.5ex>[r]_{p_2}\ar@<.5ex>[r]^{p_1}&X}$$
In order to show that $p_1=p_2: T\rightrightarrows X$, it suffices to show that $p_1=p_2: \Spec\mathcal{O}_{T,t}\rightrightarrows\Spec\mathcal{O}_{X,x}$. This is equivalent to show that $p^\sharp_1=p^\sharp_2: \mathcal{O}_{X,x}\rightrightarrows\mathcal{O}_{T,t}$. Consider the commutative diagram
$$\xymatrix {\widehat{\mathcal{O}}_{Y,f(x)}\ar[r]^{\widehat{f}^\sharp}& \widehat{\mathcal{O}}_{X,x}\ar@<-.5ex>[r]_{\widehat{p}^\sharp_2}\ar@<.5ex>[r]^{\widehat{p}^\sharp_1}&\widehat{\mathcal{O}}_{T,t}\\
&\mathcal{O}_{X,x}\ar[u]^{\sigma_{X,x}}\ar@<-.5ex>[r]_{p^\sharp_2}\ar@<.5ex>[r]^{p^\sharp_1}&\mathcal{O}_{T,t}\ar[u]_{\sigma_{T,t}} }
$$
Since $T$ is an integral Noetherian scheme , then $\mathcal{O}_{T,t}$ is a Noetherian integral domain. So the rightmost map $\sigma_{T,t}:\mathcal{O}_{T,t}\longrightarrow\widehat{\mathcal{O}}_{T,t}$ is injective since $\ker (\sigma_{T,t})=\bigcap_n\mathfrak{m}_{T,t}^n=0$ by \cite[Theorem 8.10(ii)]{Matsumura}. Hence it will suffice to show that $\widehat{p}^\sharp_1=\widehat{p}^\sharp_2$. The condition $f\circ p_1=f\circ p_2$ implies
$\widehat{p}^\sharp_1\circ\widehat{f}^\sharp=\widehat{p}^\sharp_2\circ\widehat{f}^\sharp$. And $\widehat{f}^\sharp:\widehat{\mathcal{O}}_{Y,f(x)}\longrightarrow\widehat{\mathcal{O}}_{X,x}$ is surjective since $f$ is a formal immersion at $x$. Therefore $\widehat{p}^\sharp_1=\widehat{p}^\sharp_2$.
\end{proof}

Assume that $N$ is large enough so that $Gon(X_0(N))>d$. Then by Lemma \ref{Frey}, we may define an embedding $\Phi:X_0(N)^{(d)}\hookrightarrow J_0(N)$. We compose this with the natural projection $J_0(N)\longrightarrow J_0^e$ to obtain a map $f:X_0(N)^{(d)}\longrightarrow J_0^e$. Denote $S'=\Spec\mathbb{Z}[1/N]$. Since $X_0(N)$ is a smooth scheme over $S'$, then $X_0(N)^{(d)}$ is also a smooth scheme over $S'$. Since $J_0^e$ is an abelian variety over $\mathbb{Q}$, it has a N\'{e}ron model $J^e_{0/S'}$. We also use $f$ to denote the map $f:X_0(N)^{(d)}_{/S'}\longrightarrow J^e_{0/S'}$. Parent \cite{Parent1} proved the following Kamienny's criterion.

\begin{prop}[Kamienny's criterion] \label{Kamienny} Suppose $p>2$ and $p\nmid N$. The following (1) and (2) are equivalent. Furthermore, these two conditions are satisfied if (3) is true.
\begin{enumerate}
\item The map $f:X_0(N)^{(d)}_{/S'}\longrightarrow J^e_{0/S'}$ is a formal immersion along the section $(\infty,\cdots,\infty)$ in characteristic $p$.
\item $T_1e,\cdots,T_de$ are $\mathbb{F}_p$-linearly independent in $\mathbb{T}e/p\mathbb{T}e$.
\item $T_1\{0,\infty\},\cdots, T_{sd}\{0,\infty\}$ are $\mathbb{F}_p$-linearly independent in $H_1(X_0(N),cusps,\mathbb{Z})\otimes\mathbb{F}_p$ (here $s$ is the smallest prime number not dividing $N$).
\end{enumerate}

\end{prop}

In order to apply this criterion in our cases, we need the following Lemma \ref{General} and Lemma \ref{Squarefree}.

\begin{lem}\label{General} Let $N=q_1^{e_1}\cdots q_n^{e_n}$ be a positive integer with $q_1,\cdots, q_n$ distinct prime numbers. Let $p\nmid N$ be a prime number with $N>(1+\sqrt{p^d})^2$ and $q_j^{e_j}\nmid p^{2i}-1$, for all $1\leq j\leq n$ and all $1\leq i\leq d$. Suppose that $E$ is an elliptic curve over a number field $K$ of degree $d$ with $P$ a $K$-rational point of order $N$, i.e. $(E,\pm P)\in Y_1(N)(K)$. Let $x=\pi(E,\pm P)$ be the projection of $(E,\pm P)$ on $Y_0(N)(K)$. Let $\wp$ be a prime of $\mathcal{O}_K$ above $p$ and let $k$ be the residue field of $\wp$. Then $x_{1/\tau_1(\wp)}=\cdots=x_{d/\tau_d(\wp)}=\infty_{/\wp}$.
\end{lem}

\begin{proof} Let $(\widetilde{E},\widetilde{P})$ be the reduction of $(E,P)$. It suffices to verify that $E$ has multiplicative reduction at $\wp$ and $\pi(E,\pm P)$ specialize to $\infty$.

If $E$  has good reduction at $\wp$, then $\widetilde{E}$ is an elliptic curve with a $k$-rational point $\widetilde{P}$ of order $N$. By the Hasse's theorem, $\widetilde{E}(k)$ has order at most $(1+\sqrt{p^d})^2$. This is impossible under our assumption of $N$.

If $E$ has additive reduction at $\wp$, then $\widetilde{E}(k)^0\cong\mathbb{G}_{a/k}$ with $|\mathbb{G}_{a/k}|=p^i, i\leq d$ and $\widetilde{E}(k)/\widetilde{E}(k)^0\cong G$ with $|G|\leq4$.
Since $\widetilde{P}$ is a $k$-rational point of order $N$ in $\widetilde{E}$, then $N$ divides $|\widetilde{E}(k)|=|\mathbb{G}_{a/k}||G|$, which is impossible under our assumption.

So $E$ has multiplicative reduction at $\wp$, then over an quadratic extension $\mathscr{K}$ of $k$, we have an isomorphism $\widetilde{E}(\mathscr{K})^0\cong\mathbb{G}_{m/\mathscr{K}}$.

Suppose $(E,P)$ specialize to $(C_n,(\zeta_N^r,b)))$ where $C_n$ is a N\'{e}ron $n$-gon with $n<N$  and $(\zeta_N^r,b)$ is a point of order $N$ in the smooth locus $C_n^{sm}=\mathbb{G}_{m/\mathscr{K}}\times\mathbb{Z}/n\mathbb{Z}$. Then the order of $b$ in $\mathbb{Z}/d\mathbb{Z}$ is $\leq n <N$. Therefore, for a prime $q_j|(N/n)$, one has that $(N/q_j) P$ specialize into the identity component $\widetilde{E}(\mathscr{K})^0\cong\mathbb{G}_{m/\mathscr{K}}$.

Now consider the point $P' := (N/q_j^{e_j}) P$, whose specialization is of order $q_j^{e_j}$ on $\mathbb{G}_{m/\mathscr{K}}\times \mathbb{Z}/n\mathbb{Z}$. Write $\widetilde{P}' = (\widetilde{P}'_1,\widetilde{P}'_2)$ with $\widetilde{P}'_1$ a point of $\mathbb{G}_{m/\mathscr{K}}$ and $\widetilde{P}'_2$ a point of  $\mathbb{Z}/n\mathbb{Z}$. The fact that $q_j^{e_j-1}P'=(N/q_j) P$ specializes into the identity component means that $\widetilde{P}'_2$ has order dividing $q_j^{e_j-1}$. So, the only possibility for $\widetilde{P}'$ to have order $q_j^{e_j}$ is then for $\widetilde{P}'_1$ to be of that order. So $q_j^{e_j}$ must divide the cardinality of $\mathscr{K}^*$, which itself must divide $p^{2i}-1$, where $i$ is the degree of $k$ over $\mathbb{F}_p$. This contradicts our assumption of $N$.

So $(E,P)$ must specialize to to $(C_n,(\zeta_N^r,b)))$ where $C_n$ is the N\'{e}ron $n$-gon with $n=N$  and $(\zeta_N^r,b)$ is a point of order $N$ in the smooth locus $C_N^{sm}=\mathbb{G}_{m/\mathscr{K}}\times\mathbb{Z}/N\mathbb{Z}$. Hence $\pi(E,\pm P)$ specialize to $\infty:=(C_N,\mathbb{Z}/N\mathbb{Z})$.
\end{proof}

\begin{thm}\label{T3} Let $N=q_1^{e_1}\cdots q_n^{e_n}$ be an odd positive integer such that $Gon(X_0(N))>d$. Suppose there is a prime $p\nmid N, p>2$ satisfying:
\begin{enumerate}
\item $N>(1+\sqrt{p^d})^2$ and $q_j^{e_j}\nmid p^{2i}-1$, for all $1\leq j\leq n$ and all $1\leq i\leq d$.
\item $T_1\{0,\infty\},\cdots, T_{2d}\{0,\infty\}$ are linearly independent mod $p$ in $H_1(X_0(N),cusps,\mathbb{Z})$.
\end{enumerate}
Then for any elliptic curve $E$ defined over a number field $K$ with $[K:\mathbb{Q}]=d$, the cyclic group $\mathbb{Z}/N\mathbb{Z}$ is not a subgroup of $E(K)_{tor}$.
\end{thm}

\begin{proof} The proof of this theorem is in essence the same as that of Theorem 3.3 in Kamienny \cite{Kamienny2}. Suppose we have a number field $K$ with $[K:\mathbb{Q}]=d$ and an elliptic curve $E$ defined over $K$ such that $\mathbb{Z}/N\mathbb{Z}\subseteq E(K)_{tor}$. Take a generator $P$ of $\mathbb{Z}/N\mathbb{Z}$, then we have a noncuspidal point $x=\pi(E,\pm P)$ on $X_0(N)$ of degree $d$. By Lemma \ref{General}, we have that the $S$-sections $(x_1,\cdots,x_d)$ and $(\infty,\cdots,\infty)$ meet at the prime $p$. Consequently, we have that $f(x_1,\cdots,x_d)_{/p}=f(\infty,\cdots,\infty)_{/p}$. However, the points $f(x_1,\cdots,x_d)$ and $f(\infty,\cdots,\infty)$ are both $\mathbb{Q}$-rational. And we know $J_0^e(\mathbb{Q})$ is finite \cite[Theorem 1.7]{Parent1}. So by Lemma \ref{Katz}, the $S$-sections $f(x_1,\cdots,x_d)$ and $f(\infty,\cdots,\infty)$ coincide. And by Proposition \ref{Kamienny}, $f$ is a formal immersion at $(\infty,\cdots,\infty)_p$. Therefore by Lemma \ref{Oesterle}, the sections $(x_1,\cdots,x_d)$ and $(\infty,\cdots,\infty)$ coincide. This contradicts our assumption that $x$ is noncuspidal.
\end{proof}

In the special case when $N$ is square free, and $K$ is cubic, we can weaken the condition in Lemma \ref{General} and get the following:

\begin{lem}\label{Squarefree} Let $N$ be a square free positive integer with $g(X_0(N))>0$. Let $p\nmid N$ be a prime number with $N>(1+\sqrt{p^3})^2$ and $N$ is coprime with $p^{2}-1$. Suppose that $E$ is an elliptic curve over a cubic number field $K$ with $P$ a $K$-rational point of order $N$, i.e. $y=(E,\pm P)\in Y_1(N)(K)$. Let $x=\pi(E,\pm P)$ be the projection of $y$ on $Y_0(N)(K)$. Then there is  a prime $\wp$ of $\mathcal{O}_K$ above $p$ with residue field $k$, such that either $x_{1/\tau_1(\wp)}=\cdots=x_{3/\tau_3(\wp)}=\infty_{/\wp}$, or there is an
Atkin-Lehner involution $w_n$ on $X_0(N)$ with $w_n(x_1)_{/\tau_1(\wp)}=\cdots=w_n(x_3)_{/\tau_3(\wp)}=\infty_{/\wp}$.
\end{lem}

\begin{proof}
We can always choose $\wp$ such that the residue field $k=\mathcal{O}_K/\wp$ has degree $1$ or $3$ over $\mathbb{F}_p$. In fact, the decomposition of $p$ in $\mathcal{O}_K$ has the following five types
$$\aligned&I: &p\mathcal{O}_K&=\wp&&II: &p\mathcal{O}_K&=\wp^3& &III: &p\mathcal{O}_K&=\wp_1\wp_2&\\ &IV: &p\mathcal{O}_K&=\wp_1\wp_2^2&
 &V: &p\mathcal{O}_K&=\wp_1\wp_2\wp_3&\endaligned$$
In type $II,IV,V$, all the primes over $p$ have degree $1$ residue field. In type $I$, the prime over $p$ has degree $3$ residue field. In type $III$, the degree of the residue fields of the two primes $\wp_1,\wp_2$ is $1$ and $2$ respectively. We choose the one with degree $1$ residue field as $\wp$.

By the same reason as in the proof of Lemma \ref{General}, $E$ has multiplicative reduction at $\wp$. If the degree of $k$ over $\mathbb{F}_p$ is $1$, since we assume $N$ is coprime with $p^{2}-1$, then the same reasoning as that in the proof of Lemma \ref{General} leads to $x_{1/\tau_1(\wp)}=\cdots=x_{3/\tau_3(\wp)}=\infty_{/\wp}$.

If the degree of $k$ over $\mathbb{F}_p$ is $3$, consider the Galois closure $L$ of $K$. Then either $\Gal(L/\mathbb{Q})=\mathbb{Z}/3\mathbb{Z}$ or $\Gal(L/\mathbb{Q})=S_3$. We claim that there is an element $\sigma\in\Gal(L/\mathbb{Q})$ of order $3$, such that (after a necessary rearrangement) the embeddings $\tau_i:K\hooklongrightarrow\mathbb{C}, 1\leq i\leq 3$ satisfy
$$\tau_i=\sigma^i|_K$$

In fact, if $\Gal(L/\mathbb{Q})=\mathbb{Z}/3\mathbb{Z}$, then $L=K$, i.e. $K/\mathbb{Q}$ is a Galois extension. So $\Gal(K/\mathbb{Q})=\{\tau_1,\tau_2,\tau_3\}$. Let $\sigma$ be a generator of $\Gal(K/\mathbb{Q})=\mathbb{Z}/3\mathbb{Z}$. Then, after a necessary rearrangement of $\tau_1,\tau_2$ and $\tau_3$, we have $\tau_i=\sigma^i$.

Otherwise, if $\Gal(L/\mathbb{Q})=S_3$, then $L/K$ is a quadratic extension. There is an element $\sigma_2\in\Gal(L/\mathbb{Q})$ of order $2$ such that $\Gal(L/K)=\langle\sigma_2\rangle$. Let $\sigma_3\in\Gal(L/\mathbb{Q})$ be an element of order $3$. Then $\Gal(L/\mathbb{Q})=\langle\sigma_2,\sigma_3\rangle$. On the other hand, each $\tau_i$ extends to two embedding $\tau_{i1},\tau_{i2}:L\hooklongrightarrow\mathbb{Q}$ and $\Gal(L/\mathbb{Q})=\{\tau_{11},\tau_{12},\tau_{21},\tau_{22},\tau_{31},\tau_{32}\}$. Without loss of generality, suppose $\tau_3$ is the identity embedding, then $\{\tau_{31},\tau_{32}\}=\{id,\sigma_2\}$, and after a necessary rearrangement of $\tau_1$ and $\tau_2$, $\{\tau_{11},\tau_{12}\}=\{\sigma_3,\sigma_3\sigma_2\}, \{\tau_{21},\tau_{22}\}=\{\sigma_3^2,\sigma_3^2\sigma_2\}$. Let $\sigma=\sigma_3$. Then $\tau_i=\sigma^i|_K$.

Let $\wp'$ be a prime of $L$ over $\wp$ with residue field $k'=\mathcal{O}_L/\wp'$ (which is an extension of $k$). It is known in algebraic number theory that the Frobenius $\phi\in\Gal(k'/\mathbb{F}_p)$ is the reduction from a Frobenius element $\sigma'=Frob_{\wp'}$ in $\Gal(L/\mathbb{Q})$. It is easy to see $k'=k$ since the highest order of an element in $\Gal(L/\mathbb{Q})$ is $3$. Since the only elements of order $3$ in $\Gal(L/\mathbb{Q})$ are $\sigma$ and $\sigma^2$, then either $\sigma'=\sigma$ or $\sigma'=\sigma^2$. Without loss of generality, lets suppose $\sigma'=\sigma$. Then the following reduction diagram is commutative for all $1\leq i\leq 3$:
$$\xymatrix{X_1(N)\ar[r]^{\otimes\overline{\mathbb{F}}_p}\ar[d]^{\tau_i}&\widetilde{X}_1(N)\ar[r]^{\pi}\ar[d]^{\phi^i}&\widetilde{X}_0(N)\ar[d]^{\phi^i}\\
X_1(N)\ar[r]^{\otimes\overline{\mathbb{F}}_p}&\widetilde{X}_1(N)\ar[r]^{\pi}&\widetilde{X}_0(N)}
$$

Let $y_1,y_2,y_3$ (resp. $x_1,x_2,x_3$) be the images of $y$ (resp. $x$) under the distinct embeddings $\tau_i:K\hooklongrightarrow\mathbb{C}, 1\leq i\leq 3$. Since the action of $\Gal(\overline{\mathbb{Q}}/\mathbb{Q})$ on the covering $X_1(N)\longrightarrow X_0(N)$ is compatible, i.e. the following diagram is commutative, we have $x_i=\pi(y_i), 1\leq i\leq 3$.
$$\xymatrix {X_1(N)\ar[d]_{\pi}\ar[r]^{\tau_i}&X_1(N)\ar[d]^{\pi}\\
X_0(N)\ar[r]^{\tau_i}&X_0(N) }
$$

Let $c$ be a cusp of $X_1(N)$ such that
$$y\otimes\overline{\mathbb{F}}_p=c\otimes\overline{\mathbb{F}}_p$$
then for $1\leq i\leq 3$
$$y_i\otimes\overline{\mathbb{F}}_p=\tau_i(y)\otimes\overline{\mathbb{F}}_p=\phi^i(y\otimes\overline{\mathbb{F}}_p)
=\phi^i(c\otimes\overline{\mathbb{F}}_p)=\tau_i(c)\otimes\overline{\mathbb{F}}_p$$

We know that the action of each $\tau_i$ on the cusps factors through $\Gal(\mathbb{Q}(\zeta_N)/\mathbb{Q})\cong(\mathbb{Z}/N\mathbb{Z})^\times$. Suppose $c$ is represented by $(C_n,(\zeta_N^r,b))$, then $\Gal(\mathbb{Q}(\zeta_N)/\mathbb{Q})\cong(\mathbb{Z}/N\mathbb{Z})^\times$ acts on $c$ as
$$(C_n,(\zeta_N^r,b))^a=(C_n,(\zeta_N^{ra},b))$$
So all the $\tau_i(c)$ are in the form of $(C_n,(\zeta_N^{ra_i},b))$ for some $a_i\in(\mathbb{Z}/N\mathbb{Z})^\times$. Since $N$ is square free, we know they all maps to the unique cusp of the form $(C_n, G)$ on $X_0(N)$. Denote $(C_n,G)$ as $c_n$. Then
$$x_i\otimes\overline{\mathbb{F}}_p=\pi(y_i)\otimes\overline{\mathbb{F}}_p=\pi(y_i\otimes\overline{\mathbb{F}}_p)
=\pi(\tau_i(c)\otimes\overline{\mathbb{F}}_p)=\pi(\tau_i(c))\otimes\overline{\mathbb{F}}_p=c_n\otimes\overline{\mathbb{F}}_p$$
We know the Atkin-Lehner involutions act transitively on the cusps of $X_0(N)$ if $N$ is square free. In fact, by applying the Atkin-Lehner involution $w_n$ one gets that $w_n(c_n)=\infty$. And because the reduction diagram
$$\xymatrix {X_0(N)\ar[d]_{\otimes\overline{\mathbb{F}}_p}\ar[r]^{w_n}&X_0(N)\ar[d]^{\otimes\overline{\mathbb{F}}_p}\\
\widetilde{X}_0(N)\ar[r]^{w_n}&\widetilde{X}_0(N) }
$$
is commutative when the genus of $X_0(N)$ is positive (See Diamond-Shurman \cite{DiamondShurman} Theorem 8.5.7). So we have
$$w_n(x_i)\otimes\overline{\mathbb{F}}_p=w_n(x_i\otimes\overline{\mathbb{F}}_p)=w_n(c_n\otimes\overline{\mathbb{F}}_p)=
w_n(c_n)\otimes\overline{\mathbb{F}}_p=\infty\otimes\overline{\mathbb{F}}_p$$
i.e.
$$w_n(x_1)_{/\tau_1(\wp)}=\cdots=w_n(x_3)_{/\tau_3(\wp)}=w_n(c_n)_{/\wp}=\infty_{/\wp}.$$
\end{proof}

\begin{thm}\label{T4} Let $N$ be an odd square free positive integer such that $Gon(X_0(N))>d$ and the genus  $g(X_0(N))>0$. Suppose there is a prime $p\nmid N, p>2$ satisfying:
\begin{enumerate}
\item $N>(1+\sqrt{p^3})^2$ and $N$ is coprime with $p^{2}-1$.
\item $T_1\{0,\infty\},\cdots, T_{2d}\{0,\infty\}$ are linearly independent mod $p$ in $H_1(X_0(N),cusps,\mathbb{Z})$.
\end{enumerate}
Then for any elliptic curve $E$ defined over a cubic number field $K$, the cyclic group $\mathbb{Z}/N\mathbb{Z}$ is not a subgroup of $E(K)_{tor}$.
\end{thm}

\begin{proof} With Lemma \ref{Squarefree} at hand, the proof of this theorem is exactly the same as that of Theorem \ref{T3} except replacing $(x_1,\cdots,x_d)$ by $(w_n(x_1),\cdots,w_n(x_d))$ when necessary.
\end{proof}

\section{Proof of Theorem \ref{T1}}

The calculations in this section is done in Sage \cite{Sage}. The elements in $H_1(X_0(N),cusps,\mathbb{Z})$ can be represented by the \emph{Manin symbols} (detailed description of this treatment can be found in Stein's book \cite[\S 3]{Stein}). Under this representation, the element $\{0,\infty\}$ is represented by the Manin symbol $(0,1)$. By Proposition 20 of Merel \cite{Merel94}, the action of Hecke operators $T_n$ on Manin symbols can be calculated by the formula:
$$T_n(x,y)=\sum_{a>b\geq0,~~d>c\geq0,~~ad-bc=n}(x,y)\left[\left(
                         \begin{array}{cc}
                           a & b \\
                           c & d \\
                         \end{array}
                       \right)\right]=(x,y)h_n
$$
where in the sum $h_n=\sum_{a>b\geq0,~~d>c\geq0,~~ad-bc=n}\left[\left(
                         \begin{array}{cc}
                           a & b \\
                           c & d \\
                         \end{array}
                       \right)\right]$, if $(x',y')=(x,y)\left[\left(
                         \begin{array}{cc}
                           a & b \\
                           c & d \\
                         \end{array}
                       \right)\right]\in(\mathbb{Z}/N\mathbb{Z})^2$ and $\gcd(x',y',N)\neq1$, then we omit the corresponding summand.

When $n$ is small enough such that $\gcd(x',y',N)=1$ for all summands, the formula is independent of the level $N$. Under this assumption, the first six $h_n$'s are
$$\aligned h_1=&\left[\left(
                         \begin{array}{cc}
                           1 & 0 \\
                           0 & 1 \\
                         \end{array}
                       \right)\right]
\\h_2=&\left[\left(
                         \begin{array}{cc}
                           1 & 0 \\
                           0 & 2 \\
                         \end{array}
                       \right)\right]+\left[\left(
                         \begin{array}{cc}
                           1 & 0 \\
                           1 & 2 \\
                         \end{array}
                       \right)\right]+\left[\left(
                         \begin{array}{cc}
                           2 & 0 \\
                           0 & 1 \\
                         \end{array}
                       \right)\right]+\left[\left(
                         \begin{array}{cc}
                           2 & 1 \\
                           0 & 1 \\
                         \end{array}
                       \right)\right]
\\h_3=&\left[\left(
                         \begin{array}{cc}
                           1 & 0 \\
                           0 & 3 \\
                         \end{array}
                       \right)\right]+\left[\left(
                         \begin{array}{cc}
                           1 & 0 \\
                           1 & 3 \\
                         \end{array}
                       \right)\right]+\left[\left(
                         \begin{array}{cc}
                           1 & 0 \\
                           2 & 3 \\
                         \end{array}
                       \right)\right]+\left[\left(
                         \begin{array}{cc}
                           3 & 0 \\
                           0 & 1 \\
                         \end{array}
                       \right)\right]+\left[\left(
                         \begin{array}{cc}
                           3 & 1 \\
                           0 & 1 \\
                         \end{array}
                       \right)\right]+\left[\left(
                         \begin{array}{cc}
                           3 & 2 \\
                           0 & 1 \\
                         \end{array}
                       \right)\right]\\&+\left[\left(
                         \begin{array}{cc}
                           2 & 1 \\
                           1 & 2 \\
                         \end{array}
                       \right)\right]
\endaligned$$
$$\aligned h_4=&\left[\left(
                         \begin{array}{cc}
                           1 & 0 \\
                           0 & 4 \\
                         \end{array}
                       \right)\right]+\left[\left(
                         \begin{array}{cc}
                           1 & 0 \\
                           1 & 4 \\
                         \end{array}
                       \right)\right]+\left[\left(
                         \begin{array}{cc}
                           1 & 0 \\
                           2 & 4 \\
                         \end{array}
                       \right)\right]+\left[\left(
                         \begin{array}{cc}
                           1 & 0 \\
                           3 & 4 \\
                         \end{array}
                       \right)\right]+\left[\left(
                         \begin{array}{cc}
                           4 & 0 \\
                           0 & 1 \\
                         \end{array}
                       \right)\right]+\left[\left(
                         \begin{array}{cc}
                           4 & 1 \\
                           0 & 1 \\
                         \end{array}
                       \right)\right]\\&+\left[\left(
                         \begin{array}{cc}
                           4 & 2 \\
                           0 & 1 \\
                         \end{array}
                       \right)\right]+\left[\left(
                         \begin{array}{cc}
                           4 & 3 \\
                           0 & 1 \\
                         \end{array}
                       \right)\right]+\left[\left(
                         \begin{array}{cc}
                           2 & 0 \\
                           0 & 2 \\
                         \end{array}
                       \right)\right]+\left[\left(
                         \begin{array}{cc}
                           2 & 0 \\
                           1 & 2 \\
                         \end{array}
                       \right)\right]+\left[\left(
                         \begin{array}{cc}
                           2 & 1 \\
                           0 & 2 \\
                         \end{array}
                       \right)\right]+\left[\left(
                         \begin{array}{cc}
                           2 & 1 \\
                           2 & 3 \\
                         \end{array}
                       \right)\right]\\&+\left[\left(
                         \begin{array}{cc}
                           3 & 2 \\
                           1 & 2 \\
                         \end{array}
                       \right)\right]
\endaligned$$
$$\aligned h_5=&\left[\left(
                         \begin{array}{cc}
                           1 & 0 \\
                           0 & 5 \\
                         \end{array}
                       \right)\right]+\left[\left(
                         \begin{array}{cc}
                           1 & 0 \\
                           1 & 5 \\
                         \end{array}
                       \right)\right]+\left[\left(
                         \begin{array}{cc}
                           1 & 0 \\
                           2 & 5 \\
                         \end{array}
                       \right)\right]+\left[\left(
                         \begin{array}{cc}
                           1 & 0 \\
                           3 & 5 \\
                         \end{array}
                       \right)\right]+\left[\left(
                         \begin{array}{cc}
                           1 & 0 \\
                           4 & 5 \\
                         \end{array}
                       \right)\right]+\left[\left(
                         \begin{array}{cc}
                           5 & 0 \\
                           0 & 1 \\
                         \end{array}
                       \right)\right]\\&+\left[\left(
                         \begin{array}{cc}
                           5 & 1 \\
                           0 & 1 \\
                         \end{array}
                       \right)\right]+\left[\left(
                         \begin{array}{cc}
                           5 & 2 \\
                           0 & 1 \\
                         \end{array}
                       \right)\right]+\left[\left(
                         \begin{array}{cc}
                           5 & 3 \\
                           0 & 1 \\
                         \end{array}
                       \right)\right]+\left[\left(
                         \begin{array}{cc}
                           5 & 4 \\
                           0 & 1 \\
                         \end{array}
                       \right)\right]+\left[\left(
                         \begin{array}{cc}
                           2 & 1 \\
                           1 & 3 \\
                         \end{array}
                       \right)\right]+\left[\left(
                         \begin{array}{cc}
                           3 & 1 \\
                           1 & 2 \\
                         \end{array}
                       \right)\right]\\&+\left[\left(
                         \begin{array}{cc}
                           2 & 1 \\
                           3 & 4 \\
                         \end{array}
                       \right)\right]+\left[\left(
                         \begin{array}{cc}
                           4 & 3 \\
                           1 & 2 \\
                         \end{array}
                       \right)\right]+\left[\left(
                         \begin{array}{cc}
                           3 & 2 \\
                           2 & 3 \\
                         \end{array}
                       \right)\right]
\endaligned$$
$$\aligned h_6=&\left[\left(
                         \begin{array}{cc}
                           1 & 0 \\
                           0 & 6 \\
                         \end{array}
                       \right)\right]+\left[\left(
                         \begin{array}{cc}
                           1 & 0 \\
                           1 & 6 \\
                         \end{array}
                       \right)\right]+\left[\left(
                         \begin{array}{cc}
                           1 & 0 \\
                           2 & 6 \\
                         \end{array}
                       \right)\right]+\left[\left(
                         \begin{array}{cc}
                           1 & 0 \\
                           3 & 6 \\
                         \end{array}
                       \right)\right]+\left[\left(
                         \begin{array}{cc}
                           1 & 0 \\
                           4 & 6 \\
                         \end{array}
                       \right)\right]+\left[\left(
                         \begin{array}{cc}
                           1 & 0 \\
                           5 & 6 \\
                         \end{array}
                       \right)\right]\\&+\left[\left(
                         \begin{array}{cc}
                           6 & 0 \\
                           0 & 1 \\
                         \end{array}
                       \right)\right]+\left[\left(
                         \begin{array}{cc}
                           6 & 1 \\
                           0 & 1 \\
                         \end{array}
                       \right)\right]+\left[\left(
                         \begin{array}{cc}
                           6 & 2 \\
                           0 & 1 \\
                         \end{array}
                       \right)\right]+\left[\left(
                         \begin{array}{cc}
                           6 & 3 \\
                           0 & 1 \\
                         \end{array}
                       \right)\right]+\left[\left(
                         \begin{array}{cc}
                           6 & 4 \\
                           0 & 1 \\
                         \end{array}
                       \right)\right]+\left[\left(
                         \begin{array}{cc}
                           6 & 5 \\
                           0 & 1 \\
                         \end{array}
                       \right)\right]\\&+\left[\left(
                         \begin{array}{cc}
                           2 & 0 \\
                           0 & 3 \\
                         \end{array}
                       \right)\right]+\left[\left(
                         \begin{array}{cc}
                           2 & 0 \\
                           1 & 3 \\
                         \end{array}
                       \right)\right]+\left[\left(
                         \begin{array}{cc}
                           2 & 0 \\
                           2 & 3 \\
                         \end{array}
                       \right)\right]+\left[\left(
                         \begin{array}{cc}
                           2 & 1 \\
                           0 & 3 \\
                         \end{array}
                       \right)\right]+\left[\left(
                         \begin{array}{cc}
                           3 & 0 \\
                           0 & 2 \\
                         \end{array}
                       \right)\right]+\left[\left(
                         \begin{array}{cc}
                           3 & 1 \\
                           0 & 2 \\
                         \end{array}
                       \right)\right]\\&+\left[\left(
                         \begin{array}{cc}
                           3 & 2 \\
                           0 & 2 \\
                         \end{array}
                       \right)\right]+\left[\left(
                         \begin{array}{cc}
                           3 & 0 \\
                           1 & 2 \\
                         \end{array}
                       \right)\right]+\left[\left(
                         \begin{array}{cc}
                           2 & 1 \\
                           2 & 4 \\
                         \end{array}
                       \right)\right]+\left[\left(
                         \begin{array}{cc}
                           4 & 2 \\
                           1 & 2 \\
                         \end{array}
                       \right)\right]+\left[\left(
                         \begin{array}{cc}
                           2 & 1 \\
                           4 & 5 \\
                         \end{array}
                       \right)\right]+\left[\left(
                         \begin{array}{cc}
                           5 & 4 \\
                           1 & 2 \\
                         \end{array}
                       \right)\right]\\&+\left[\left(
                         \begin{array}{cc}
                           3 & 2 \\
                           3 & 4 \\
                         \end{array}
                       \right)\right]+\left[\left(
                         \begin{array}{cc}
                           4 & 3 \\
                           2 & 3 \\
                         \end{array}
                       \right)\right]
                       \endaligned$$

~

For $N=169,143,91,65,77,55$, the Manin basis of $H_1(X_0(N),cusps,\mathbb{Z})$ are listed in Table \ref{ManinBasis}. And the actions of $T_1,\cdots, T_6$ on the Manin symbol $(0,1)$ in terms of the Manin basis are given in Table \ref{HeckeOperators}.

\begin{table}[!ht]
\tabcolsep 0pt
\vspace*{0pt}
\begin{center}
\def\temptablewidth{1\textwidth}
\setlength{\abovecaptionskip}{0pt}
\setlength{\belowcaptionskip}{-5pt}
\caption{Manin basis of $H_1(X_0(N),cusps,\mathbb{Z})$}
\label{ManinBasis}
{\rule{\temptablewidth}{1pt}}
\begin{tabular*}{\temptablewidth}{@{\extracolsep{\fill}}ccccccccccccc}
$N$&$\dim$&Manin basis of $H_1(X_0(N),cusps,\mathbb{Z})$\\\hline
$169$&$29$&$(1,0),(1,133),(1,134),(1,135),(1,138),(1,139),(1,151),(1,152),(1,153),(1,158),$\\
&&$(1,159),(1,160),(1,163),(1,164),(1,165),(1,166),(1,167),(13,1),(13,2),(13,3),$\\
&&$(13,4),(13,5),(13,6),(13,7),(13,8),(13,9),(13,10),(13,11),(13,12)$\\
$143$&$29$&$(1,0),(1,83),(1,113),(1,127),(1,128),(1,135),(1,139),(1,140),(1,141),(11,3),$\\
&&$(11,4),(11,5),(11,6),(11,7),(11,8),(11,9),(11,10),(11,12),(13,1),(13,2),$\\
&&$(13,3),(13,4),(13,5),(13,6),(13,7),(13,8),(13,9),(13,10),(13,11)$\\
$91$&$17$&$(1,0),(7,1),(7,2),(7,4),(7,5),(7,8),(7,9),(7,10),(7,11),(7,12),$\\
&&$(13,1),(13,2),(13,3),(13,4),(13,5),(13,6),(13,7)$\\
$65$&$13$&$(1,0),(5,2),(5,3),(5,7),(5,9),(5,11),(5,12),(5,23),(13,1),(13,2),$\\
&&$(13,3),(13,4),(13,5)$\\
$77$&$17$&$(1,0),(1,74),(1,75),(7,1),(7,3),(7,5),(7,6),(7,8),(7,9),(7,10),$\\
&&$(11,1),(11,2),(11,3),(11,4),(11,5),(11,6),(11,7)$\\
$55$&$13$&$(1,0),(1,48),(5,2),(5,4),(5,7),(5,8),(5,9),(5,21),(11,1),(11,2),$\\
&&$(11,3),(11,4),(11,5)$\\

\end{tabular*}
{\rule{\temptablewidth}{1pt}}
\end{center}
\end{table}

\begin{table}[!ht]
\tabcolsep 0pt
\vspace*{0pt}
\begin{center}
\def\temptablewidth{1\textwidth}
\setlength{\abovecaptionskip}{0pt}
\setlength{\belowcaptionskip}{-5pt}
\caption{Hecke operators on $\{0,\infty\}$ in terms of Manin symbols}
\label{HeckeOperators}
{\rule{\temptablewidth}{1pt}}
\begin{tabular*}{\temptablewidth}{@{\extracolsep{\fill}}ccccccccccccc}
$N$&$T_i$&$T_i(0,1)$\\\hline
$169$&$T_1$&$(-1,0,0,0,0,0,0,0,0,0,0,0,0,0,0,0,0,0,0,0,0,0,0,0,0,0,0,0,0)$\\
&$T_2$&$(-3,0,0,0,0,0,0,0,0,0,0,0,0,0,0,0,1,0,0,0,0,0,0,0,0,0,0,0,0)$\\
&$T_3$&$(-4,0,1,-1,1,-1,0,1,0,0,-1,0,0,0,1,0,1,0,0,0,0,0,0,0,0,0,0,0,0)$\\
&$T_4$&$(-7,0,0,-1,1,-1,0,0,0,0,0,0,-1,1,1,0,2,0,0,0,0,0,0,0,0,0,0,0,0)$\\
&$T_5$&$(-6,0,0,-2,0,0,0,1,0,0,-1,0,0,2,0,0,1,0,0,0,0,0,0,0,0,0,0,0,0)$\\
&$T_6$&$(-12,0,3,-3,2,-2,0,2,0,0,-2,0,1,0,2,0,4,0,0,0,0,0,0,0,0,0,0,0,0)$\\
$143$&$T_1$&$(-1,0,0,0,0,0,0,0,0,0,0,0,0,0,0,0,0,0,0,0,0,0,0,0,0,0,0,0,0)$\\
&$T_2$&$(-3,0,0,0,0,0,0,0,1,0,0,0,0,0,0,0,0,0,0,0,0,0,0,0,0,0,0,0,0)$\\
&$T_3$&$(-4,0,0,0,0,0,0,0,1,1,0,0,0,0,0,0,1,-1,-1,-1,1,0,0,0,0,1,-1,0,1)$\\
&$T_4$&$(-7,0,0,0,1,0,1,-1,3,1,0,0,0,-1,0,1,1,-1,0,-1,0,0,0,0,0,1,-1,0,1)$\\
&$T_5$&$(-6,0,0,0,1,0,1,-1,1,1,-1,2,0,-1,2,0,1,-3,-2,-1,0,-1,2,2,-1,1,-1,0,1)$\\
&$T_6$&$(-12,0,0,0,1,0,1,-1,4,3,-1,0,1,0,0,0,3,-4,-3,-2,2,-1,1,1,-1,3,-2,0,2)$\\
$91$&$T_1$&$(-1,0,0,0,0,0,0,0,0,0,0,0,0,0,0,0,0)$\\
&$T_2$&$(-3,0,1,-1,-1,-1,-1,1,1,0,0,0,1,0,0,0,-1)$\\
&$T_3$&$(-4,-1,1,-2,0,0,-2,2,1,-1,1,-1,2,0,-1,1,-2)$\\
&$T_4$&$(-7,-1,3,-3,-2,-2,-3,2,3,-1,1,-1,3,1,-1,1,-4)$\\
&$T_5$&$(-6,-1,1,-2,-1,-1,-2,2,1,-1,1,0,2,0,0,1,-4)$\\
&$T_6$&$(-12,-3,4,-6,-2,-2,-6,6,4,-3,3,-2,6,0,-2,3,-8)$\\
$65$&$T_1$&$(-1,0,0,0,0,0,0,0,0,0,0,0,0)$\\
&$T_2$&$(-3,1,-1,0,1,0,-1,-1,-1,0,1,1,-1)$\\
&$T_3$&$(-4,1,-1,0,2,-1,-2,-1,-2,0,2,2,-2)$\\
&$T_4$&$(-7,3,-2,0,4,-1,-4,-2,-4,-1,3,4,-2)$\\
&$T_5$&$(-5,2,-2,0,2,-2,-4,-2,-3,-1,3,3,-1)$\\
&$T_6$&$(-12,5,-4,0,6,-2,-7,-4,-6,-1,6,7,-6)$\\
$77$&$T_1$&$(-1,0,0,0,0,0,0,0,0,0,0,0,0,0,0,0,0)$\\
&$T_2$&$(-3,0,1,0,0,0,0,0,0,0,0,0,0,0,0,0,0)$\\
&$T_3$&$(-4,0,1,0,0,1,1,0,-1,0,-1,0,0,0,1,-1,1)$\\
&$T_4$&$(-7,0,3,0,-1,1,1,-1,-1,0,-1,0,1,1,1,-1,-1)$\\
&$T_5$&$(-6,0,2,-1,0,2,2,0,-2,-1,-1,0,0,0,2,-1,0)$\\
&$T_6$&$(-12,0,4,-1,0,3,3,0,-2,-1,-2,0,0,0,2,-2,2)$\\
$55$&$T_1$&$(-1,0,0,0,0,0,0,0,0,0,0,0,0)$\\
&$T_2$&$(-3,0,0,1,1,-1,0,-1,-1,1,0,0,0)$\\
&$T_3$&$(-4,0,-1,2,2,0,-1,-2,-2,1,1,0,0)$\\
&$T_4$&$(-7,0,0,3,3,-2,0,-4,-3,2,0,1,0)$\\
&$T_5$&$(-5,0,-1,2,2,-2,-1,-4,-3,2,0,1,1)$\\
&$T_6$&$(-12,0,-2,6,6,-2,-2,-8,-6,4,2,2,-2)$\\

\end{tabular*}
{\rule{\temptablewidth}{1pt}}
\end{center}
\end{table}

\subsection{$N=169$}

It is seen in Table \ref{HeckeOperators} that $T_1\{0,\infty\},\cdots, T_6\{0,\infty\}$ are linearly independent mod $5$. By Proposition \ref{Ogg} and \ref{HasegawaShimura}, we know $Gon(X_0(169))>3$. Since $169\nmid 5^{2i}-1, i=1,2,3$, and $169>(1+\sqrt{5^3})^2$, then by Theorem \ref{T3}, $\mathbb{Z}/169\mathbb{Z}$ is not a subgroup of $E(K)_{tor}$.

\subsection{$N=143,91,65,77,55$}

It is seen in Table \ref{HeckeOperators} that $T_1\{0,\infty\},\cdots, T_6\{0,\infty\}$ are linearly independent mod $3$. By Proposition \ref{Ogg} and \ref{HasegawaShimura}, we know $Gon(X_0(N))>3$. Since $(N, 3^{2}-1)=1$ and $N>(1+\sqrt{3^3})^2$, then by Theorem \ref{T4}, $\mathbb{Z}/N\mathbb{Z}$ is not a subgroup of $E(K)_{tor}$.

\subsection*{Acknowledgements}
It is a pleasure to thank my PhD advisor Sheldon Kamienny for introducing this research topic and for providing many valuable ideas and insightful comments throughout the research. I wish to thank Maarten Derickx for correcting mistakes in Lemma \ref{General} and Lemma \ref{Squarefree}. I also wish to thank Andrew Sutherland and Andreas Schweizer for pointing out several errors in an earlier version of this paper.

\end{document}